\documentclass[11pt,a4paper]{article}
\usepackage[T1]{fontenc}
\usepackage{lmodern,amsmath,amssymb,amsthm,booktabs,array}
\usepackage[margin=1in]{geometry}
\usepackage{microtype}
\usepackage[colorlinks=true,allcolors=blue]{hyperref}
\usepackage{xurl}
\newtheorem{theorem}{Theorem}[section]
\newtheorem{lemma}[theorem]{Lemma}
\newtheorem{proposition}[theorem]{Proposition}

\theoremstyle{definition}

\theoremstyle{remark}
\newtheorem{remark}[theorem]{Remark}
\newcommand{\bd}{\partial}
\newcommand{\card}[1]{\lvert#1\rvert}
\newcommand{\cl}{\operatorname{cl}}
\newcommand{\ZF}{Z}
\title{Zero forcing in generalized Petersen graphs: improved stabilization and sharp small-step thresholds}
\author{Arnav Krishnan\thanks{\href{mailto:arnav.s.krishnan@gmail.com}{arnav.s.krishnan@gmail.com}}\and
Haobo Ma\thanks{ChronoAI Pte Ltd, Singapore. \href{mailto:auric@aelf.io}{auric@aelf.io};
\href{https://orcid.org/0009-0008-4944-977X}{ORCID 0009-0008-4944-977X}.}\and
Wenlin Zhang\thanks{National University of Singapore; ChronoAI Pte Ltd, Singapore.
\href{mailto:e1327962@u.nus.edu}{e1327962@u.nus.edu};
\href{https://orcid.org/0009-0005-2414-7968}{ORCID 0009-0005-2414-7968}.}}
\date{}
\hypersetup{pdftitle={Zero forcing in generalized Petersen graphs: improved stabilization and sharp small-step thresholds},pdfauthor={Arnav Krishnan, Haobo Ma, Wenlin Zhang}}
\begin{document}
\maketitle
\begin{abstract}
The zero forcing number of the generalized Petersen graph $P(n,k)$ is
$2k+2$ for every integer $k\ge2$ and $n\ge2k^2+1$, improving the
sufficient threshold $2k^2+k+1$ of Kim, Kim and Lee without a coprimality
assumption. We bound the span of connected stalled sets containing an
outer edge. The remaining exceptional lifts lie on a single inner track
with its spokes and project into a proper stalled set. For $k=6$, the
least stabilization threshold lies between $32$ and $73$. We also give
boundary and fort proofs of the sharp thresholds $13$ for $k=3$ and $18$
for $k=4$, including exact zero forcing numbers below those thresholds.
\end{abstract}

\section{Introduction}
For a finite simple graph $G$, start with a set $S$ of black vertices and
color the remaining vertices white. A black vertex with exactly one white
neighbor may force that neighbor black. The set $S$ is a \emph{zero
forcing set} if this rule colors the whole graph, and the \emph{zero
forcing number} $\ZF(G)$ is the smallest cardinality of such a set.
Write $\cl(S)$ for its closure under the rule.

For $n>2k$, the generalized Petersen graph $P(n,k)$ has vertices
$u_i,v_i$, indexed by $i\in\mathbb Z/n\mathbb Z$, and edges
\[
 u_i u_{i+1},\qquad u_i v_i,\qquad v_i v_{i+k}.
\]
Thus $P(n,k)$ is simple and cubic. We call $\{u_i,v_i\}$ column $i$.
Alameda et al.\ \cite[Theorem~2.2]{Alameda} proved the general upper
bound $\ZF(P(n,k))\le2k+2$ in 2018. Rashidi, Shajareh Poursalavati,
and Tavakkoli \cite[Theorem~2.2 and Remark~2.4]{RSPT} gave another
proof in 2020 and acknowledged the earlier result. Their asserted equality
$\ZF(P(n,3))=8$ for $n\ge12$ was corrected by Krishnan
\cite{Krishnan}, who proved $\ZF(P(12,3))=7$, identified the gap in
the lower-bound argument, and computed the values through $n=20$.
Krishnan's Proposition~4 gives an explicit eight-vertex forcing set for
all $n\ge9$; his Conjecture~5 asserts equality for every $n\ge13$.

The new contribution is a general stabilization bound that removes the
additive term $k$ from the sufficient threshold $2k^2+k+1$ of Kim, Kim and
Lee \cite[Theorem~5.4]{KKL}. Their proof studies closures in the infinite
generalized Petersen graph. We sharpen the span bound for connected
stalled sets containing outer edges, then treat the remaining sets on a
single inner track separately when a lift touches one of its translates.
No coprimality condition on $n$ and $k$ is needed. For $k=6$, this lowers
the upper bound on $N_6$ from $79$ to $73$; the lower bound $32$ is from
\cite[Proposition~5.5]{KKL}.

Our earlier preprint \cite{KMZOriginal} proved Krishnan's conjecture and
the sharp threshold for $k=4$. We include those boundary and fort proofs,
which give a different route to the small-step results independently
obtained by Kim, Kim and Lee, who also give the values for $k=5$.
\begin{theorem}\label{thm:general-main}
For every integer $k\ge2$ and every $n\ge2k^2+1$,
\[
 \ZF(P(n,k))=2k+2.
\]
Consequently, if $N_k$ is the least eventual stabilization threshold,
then $N_k\le2k^2+1$ for $k\ge2$, and $32\le N_6\le73$.
\end{theorem}
The sharp small-step values are as follows.
\begin{theorem}\label{thm:main}
For every $n\ge13$, $\ZF(P(n,3))=8$. For every $n\ge18$,
$\ZF(P(n,4))=10$. The smaller admissible values are
\[
\begin{array}{c|rrrrrr}
 n&7&8&9&10&11&12\\\hline
 \ZF(P(n,3))&6&6&6&8&7&7
\end{array}
\]
and
\[
\begin{array}{c|rrrrrrrrr}
 n&9&10&11&12&13&14&15&16&17\\\hline
 \ZF(P(n,4))&6&6&7&6&8&8&9&8&9.
\end{array}
\]
In particular, $13$ and $18$ are the least eventual stabilization thresholds.
\end{theorem}

For $7\le n\le11$, the small $k=3$ values were reported by Bjorkman
et al.\ \cite[Proposition~4.5 and Table~1, $\ell=0$]{Bjorkman}.
Krishnan \cite[Table~1]{Krishnan} reports these values together with
his correction $\ZF(P(12,3))=7$. Our small-step lower-bound argument is uniform in
$n$: it converts a forcing set into a vertex set with small external
boundary, then reduces the possible column configurations to finitely
many cases. For $k=3$ these are gap words and layer assignments; for
$k=4$ they are propositional satisfiability instances over a bounded
range of circumferences. We give the reductions and the finite predicates
explicitly. Section~\ref{sec:verification} describes their verification
and the scope of the Lean formalization. Krishnan's discussion of $k\ge4$
in \cite[Section~5]{Krishnan} asks for stabilization thresholds and uniform
lower bounds; it does not state a separate $k=4$ conjecture.

\section{An improved general stabilization bound}\label{sec:general}
We distinguish the external boundary $\bd X=N[X]\setminus X$ used below
from the \emph{internal boundary}
\[
 D(X)=\{x\in X:N(x)\not\subseteq X\}.
\]
All boundaries in this section are internal. Fix $k\ge2$ and let
$P(\infty,k)$ have vertices $u_i,v_i$ for $i\in\mathbb Z$, with edges
$u_iu_{i+1}$, $u_iv_i$ and $v_iv_{i+k}$. The column of either $u_i$ or
$v_i$ is $i$. A set is \emph{stalled} if no black vertex has exactly one
white neighbor. For finite nonempty $X$, write
$\sigma(X)=\max_{x\in X}\operatorname{col}(x)-\min_{x\in X}\operatorname{col}(x)$.
Thus its width is $\sigma(X)+1$.

We use the following results of Kim, Kim and Lee \cite{KKL}.
The closure of at most $2k+1$ vertices in $P(\infty,k)$ is finite
(Lemma~4.4 and Proposition~5.3). A finite seed set $T$ satisfies
$|D(\cl(T))|\le|T|$ (Proposition~2.7). A finite connected stalled set
has successive internal-boundary columns separated by at most $k$,
and its extreme columns meet its internal boundary
(Lemma~5.2 and the proof of Lemma~4.7). Hence, if $b=|D(X)|$,
\begin{equation}\label{eq:span-old}
 \sigma(X)\le k(b-1).
\end{equation}
For clarity, the gap assertion follows because a gap of more than $k$
between consecutive boundary columns contains a boundary-free window of
$k$ columns. If $X$ meets that window, the stalled-set window lemma
contradicts finiteness; otherwise no edge crosses the window, contradicting
connectedness. A vertex in an extreme column has a same-layer neighbor
beyond that extreme, so it belongs to $D(X)$.

\begin{samepage}
\begin{lemma}[A span deficit for outer edges]\label{lem:span-deficit}
Let $X$ be a finite connected stalled subset of $P(\infty,k)$ with at
least three vertices. If $X$ contains an outer edge, then
\begin{equation}\label{eq:span-new}
 \sigma(X)\le k\bigl(|D(X)|-2\bigr).
\end{equation}
\end{lemma}
\end{samepage}
\begin{proof}
Put $b=|D(X)|$, $\sigma=\sigma(X)$ and
$\delta=k(b-1)-\sigma\ge0$. Suppose, for a contradiction, that
$\delta<k$. List the distinct boundary columns as
$d_1<\cdots<d_m$, and put $g_j=d_{j+1}-d_j$.
By the preceding gap assertion, $1\le g_j\le k$, and
\[
 \delta=k(b-m)+\sum_{j=1}^{m-1}(k-g_j).
\]
It follows that $b=m$: each boundary column contains exactly one
boundary vertex. Put $e_j=k-g_j$ and
$E_j=\sum_{h<j}e_h$. The integers $E_j$ are nondecreasing and range
over an interval of length less than $k$, while
\[
 d_j\equiv d_1-E_j\pmod k.
\]
Consequently, all boundary columns of any fixed residue modulo $k$
form a single contiguous block in the ordered list. For each represented
residue $r$, let $I_r$ be the interval from its first to its last boundary
column. These intervals are pairwise disjoint and ordered.

Every vertex of $X$ in residue $r$ has its column in $I_r$.
Indeed, a black inner vertex belongs to a finite component of black
vertices on its inner track. The two endpoints of that component are
boundary vertices, or its sole vertex is a boundary vertex. They bracket
the given inner vertex in residue $r$. A black outer vertex is either
itself in $D(X)$, or has its spoke neighbor black and is bracketed in the
same way.

An outer edge of $X$ joins distinct residues and therefore distinct
intervals $I_r$. Its columns are consecutive. Since the intervals are
disjoint and contain their respective endpoints, these two columns must
be successive boundary columns, with gap one. Different outer edges give
different such gaps. Each contributes $k-1$ to $\delta$. Because
$2(k-1)\ge k>\delta$, there is at most one outer edge in $X$.
If $u_i u_{i+1}$ is that edge, each endpoint has its other outer neighbor
white. Its spoke neighbor must be white as well: a black spoke neighbor
would give exactly two black neighbors and hence exactly one white
neighbor, contrary to stalledness. Thus both endpoints have degree one
in the induced subgraph on $X$. Connectedness then forces
$X=\{u_i,u_{i+1}\}$, contrary to $|X|\ge3$.
We conclude that $\delta\ge k$, which is \eqref{eq:span-new}.
\end{proof}

The bound is attained by arbitrarily long examples. For $q\ge3$, take
the inner path $v_0,v_k,\ldots,v_{(q-1)k}$, its interior spoke neighbors,
and $u_{k-1},u_{k+1}$. When $k\ge3$, this is connected and stalled,
has $q+1$ internal-boundary vertices and span $k(q-1)$.

\begin{samepage}
\begin{lemma}[Sets without outer edges]\label{lem:one-track}
A finite connected stalled set with no outer edge is a singleton, a
spoke edge, or an inner path
\[
 v_a,v_{a+k},\ldots,v_{a+(q-1)k}\qquad(q\ge2)
\]
together with exactly the spoke neighbors of its interior inner vertices.
In the path case $|D(X)|=q$, $\sigma(X)=k(q-1)$, and the span of its
outer vertices, when there are any, is at most $k(q-3)$.
\end{lemma}
\end{samepage}
\begin{proof}
Without outer edges, each outer vertex has at most its spoke neighbor
in $X$, so it cannot connect different inner components. Connectedness
therefore leaves one inner path, a spoke edge, or a singleton.
At an endpoint of a nontrivial inner path the spoke neighbor is white:
otherwise that endpoint has exactly one white neighbor. At an interior
inner vertex the spoke neighbor is black for the same reason. These are
all the outer vertices of $X$. The two inner endpoints and the $q-2$
outer leaves are its boundary, giving the stated counts.
\end{proof}

\begin{samepage}
\begin{lemma}[The self-touching case]\label{lem:self-touch}
Suppose $n\ge2k^2+1$, $|T|\le2k+1$, and $X=\cl(T)$ is nonempty
and connected in $P(\infty,k)$. If $X$ touches $X+jn$ for some nonzero
integer $j$, then
\[
 n=k(2k+1),\qquad |T|=2k+1,
\]
and $X$ is an inner path on $2k+1$ vertices with its interior spoke
neighbors. The projection of $T$ into $P(n,k)$ lies in a proper stalled
set.
\end{lemma}
\end{samepage}
\begin{proof}
Here touching means intersecting or having an edge between the sets.
We have $b=|D(X)|\le|T|\le2k+1$ and
$\sigma(X)\le2k^2<n$, so translated copies do not intersect.
If $X$ has an outer edge and at least three vertices, Lemma~\ref{lem:span-deficit}
gives $\sigma(X)\le2k^2-k<n-k$. No edge, whose column displacement
has magnitude at most $k$, can join distinct translates. Singletons and
single edges cannot touch their translates at the stated circumference.
Thus Lemma~\ref{lem:one-track} leaves an inner path with
$q=b\le2k+1$. Its outer-vertex span is at most
$k(q-3)\le2k^2-2k<n-1$, so an outer edge cannot join translates.
Spokes cannot do so either, since $\sigma(X)<n$.
An inner edge between translates requires $k\mid n$ and
$n-k\le k(q-1)$; translates with $|j|\ge2$ are too far apart.
Hence $n/k\le q\le2k+1$. Since $n>2k^2$, integrality forces
$n/k=q=2k+1$, and then $|T|=2k+1$.

Write $a$ for the initial column of the path. In the quotient, put
\[
 Y=\{v_i:i\equiv a\pmod k\}\cup\{u_i:i\equiv a\pmod k\}.
\]
Every inner vertex of $Y$ has all three neighbors in $Y$, and every
outer vertex of $Y$ has its two outer neighbors outside $Y$. Thus $Y$
is stalled. It is proper because $k\ge2$, and it contains the projected
set $T$.
\end{proof}

\begin{proof}[Proof of Theorem~\ref{thm:general-main}]
The upper bound $\ZF(P(n,k))\le2k+2$ is
\cite[Theorem~2.2]{Alameda}. Suppose that a nonempty set $S$ of at most
$2k+1$ vertices forces $P(n,k)$. Let $\pi$ denote reduction of column
indices modulo $n$. Start with one lift of each seed, as a singleton
set $T_r$. Maintain pairwise disjoint projected seed sets whose union
is $S$, with connected closures $X_r=\cl(T_r)$ in $P(\infty,k)$.
These closures are finite by \cite[Lemma~4.4 and Proposition~5.3]{KKL}.

If some $X_r$ touches one of its own nontrivial translates,
Lemma~\ref{lem:self-touch} gives $|T_r|=2k+1$. This uses the entire
seed budget, and the lemma supplies a proper stalled set containing
$S$, a contradiction. Otherwise, whenever $X_r$ touches $X_s+jn$
for $r\ne s$, replace their seed sets by $T_r\cup(T_s+jn)$.
Their projected seed sets are disjoint, so sizes add; their new closure
is connected by \cite[Lemma~2.6]{KKL}. Each merge lowers the number
of seed sets, so after finitely many merges either the preceding
self-touching case occurs, or all the translated closures are pairwise
non-touching.

In the latter case their union $U$ is stalled, by
\cite[Lemma~2.5]{KKL}, and contains $\pi^{-1}(S)$. Conversely each
translated closure belongs to $\cl(\pi^{-1}(S))$, so
$U=\cl(\pi^{-1}(S))$. The projection is a graph covering for $n>2k$;
closure commutes with full inverse image
\cite[Lemma~3.2]{KKL}. Thus
\[
 \cl(S)=\bigcup_r\pi(X_r).
\]
If there are $m\ge1$ remaining seed sets, the number of columns met by
this union is at most
\[
 \sum_r\bigl(k(|T_r|-1)+1\bigr)
 =k|S|-(k-1)m\le2k^2+1.
\]
If $n\ge2k^2+2$, a column is missed, contradicting zero forcing.
For the remaining circumference $n=2k^2+1$, the same argument still
works unless $|S|=2k+1$, $m=1$, and $\sigma(X_1)=2k^2$.
If $X_1$ contains outer edges and at least three vertices, the new
span bound excludes this equality. Otherwise Lemma~\ref{lem:one-track}
shows that $X_1$ is supported in one residue modulo $k$ and has at
most $2(2k+1)-2=4k$ vertices. Since $4k<2(2k^2+1)$ for $k\ge2$,
its projection cannot cover the graph. Singleton and single-edge cases
are smaller still. This proves the general equality.

Finally, \cite[Proposition~5.5]{KKL} gives $N_6\ge32$.
The new upper threshold is $2\cdot6^2+1=73$, giving
$32\le N_6\le73$ and $\ZF(P(n,6))=14$ for every $n\ge73$.
\end{proof}

For $k=2$ the bound gives the sharp threshold $N_2=9$ reported in
\cite[Table~1]{KKL}. For general $k\ge6$ the sharp threshold
remains open; in particular the circumferences $32\le n\le72$
are still relevant to determining $N_6$.

\section{Forcing sources, forts, and empty columns}\label{sec:reductions}
For $X\subseteq V(G)$, let
\[
 \bd X=\{w\notin X:\text{$w$ has a neighbor in $X$}\},
 \qquad N[X]=X\cup\bd X.
\]
All boundaries below are external vertex boundaries.

\begin{lemma}\label{lem:first}
If a zero forcing set $S$ in a finite graph performs at least $p$ forces,
then the sources $X$ of the first $p$ forces satisfy
$\card X=p$ and $\card{\bd X}\le\card S$.
\end{lemma}
\begin{proof}
A vertex cannot force twice: after its first force, all its neighbors
are black. After $p$ forces the black set $B$ has cardinality
$\card S+p$, contains $X$, and contains every neighbor of $X$.
Hence $\bd X\subseteq B\setminus X$ and
$\card{\bd X}\le\card B-\card X=\card S$.
\end{proof}

We use the fort terminology credited to Fast and Hicks \cite{FH}
by Brimkov, Fast, and Hicks \cite[Section~2]{BFH}, and the latter
paper's fort-cover characterization \cite[Theorem~8]{BFH}.
A nonempty set $F\subseteq V(G)$ is a \emph{fort} if every vertex outside
$F$ has either zero or at least two neighbors in $F$.
Every zero forcing set meets every fort. Indeed, the first force into a
fort disjoint from the initial black set would have a source outside
$F$ with exactly one neighbor in $F$, a contradiction. Conversely, the
nonempty white remainder of a stalled forcing process is a fort.

\begin{lemma}\label{lem:delete}
Let $k\ge1$ and $n\ge2k+3$. Suppose that columns $0,\ldots,2k$
of $P(n,k)$ contain no vertex of $X$. Delete column $k$, reindex the
surviving columns in their cyclic order, and let $X'$ be the image of
$X$ in $P(n-1,k)$. Then
$\card{X'}=\card X$ and $\card{\bd X'}=\card{\bd X}$.
\end{lemma}
\begin{proof}
On surviving column indices define
\[
 f(i)=\begin{cases}i,&0\le i<k,\\i-1,&k<i\le n-1,\end{cases}
\]
and preserve the layer. This is a bijection onto the vertices of
$P(n-1,k)$. An occupied index $i$ lies between $2k+1$ and $n-1$.
For each $d\in\{-k,-1,1,k\}$ put $q=i+d$; then
$k+1\le q\le n+k-1$. If $q<n$, its image is
$q-1=f(i)+d$. If $q\ge n$, its old residue $q-n$ lies in
$\{0,\ldots,k-1\}$ and is unchanged by $f$, while
\[
 (f(i)+d)\bmod(n-1)=q-n.
\]
Thus none of these neighbors is deleted, and each maps to the
corresponding neighbor after deletion. Conversely, an image index
$j=f(i)$ satisfies $2k\le j\le n-2$. A nonwrapping neighbor index
$j+d\ge k$ lifts to $j+d+1=i+d$, while a wrapping index has residue
in $\{0,\ldots,k-1\}$ and lifts to that same residue. Spokes are
preserved as well. It follows that $f(N[X])=N[X']$ and $f(X)=X'$.
Injectivity gives $f(\bd X)=\bd X'$ and both cardinality equalities.
\end{proof}

\section{The upper bounds}\label{sec:upper}
For $k=3$ we recall Krishnan's construction
\cite[Proposition~4]{Krishnan}. If $n\ge9$, the set
$S=\{u_0,\ldots,u_7\}$ permits the forces
\[
 u_i\longrightarrow v_i\ (1\le i\le6),\qquad
 v_4\longrightarrow v_7,\qquad u_7\longrightarrow u_8.
\]
Every source has its other two neighbors black. If $\rho$ is rotation
by one column, then $\rho(S)\subseteq\cl(S)$. Closure is monotone,
idempotent, and equivariant under automorphisms, so
\[
 \rho(\cl(S))=\cl(\rho(S))\subseteq\cl(S).
\]
The two sides have equal cardinality; hence $\cl(S)$ is rotation
invariant. It contains every outer vertex and therefore every inner
vertex by the spoke rule. This proves $\ZF(P(n,3))\le8$.

For $k=4$, an explicit instance of the construction in
\cite[Theorem~2.2]{RSPT} starts from $S=\{u_0,\ldots,u_9\}$ when $n\ge11$. First perform
\[
 u_i\longrightarrow v_i\ (1\le i\le8),\qquad
 v_4\longrightarrow v_0,\qquad v_5\longrightarrow v_9.
\]
Both vertices of columns zero through nine are now black. For
$j=10,\ldots,n-1$, perform
\begin{equation}\label{eq:propagation}
 u_{j-1}\longrightarrow u_j,\qquad
 v_{j-4}\longrightarrow v_j.
\end{equation}
Before stage $j$, all vertices in columns below $j$ are black and all
vertices in columns at least $j$ are white. The other neighbors of the
first source are $u_{j-2},v_{j-1}$; those of the second are
$u_{j-4},v_{j-8}$. All lie in earlier columns. The first force leaves
$v_j$ white, so both forces are legal and preserve the invariant.
For $n=9,10$, all outer vertices form a forcing set. Thus
$\ZF(P(n,4))\le10$ for every admissible $n$.

\section{Ten-vertex boundaries in \texorpdfstring{$P(n,3)$}{P(n,3)}}\label{sec:k3}
\begin{proposition}\label{prop:k3boundary}
If $n\ge14$ and $X\subseteq V(P(n,3))$ has ten vertices, then
$\card{\bd X}\ge8$.
\end{proposition}
The proof occupies the next two subsections.

\subsection{An exact boundary count}
Let $C$ be the set of occupied columns, put $c=\card C$, and set
$O=\{u_i,v_i:i\in C\}$. The same-layer shifts
$u_i\mapsto u_{i\pm1}$ and $v_i\mapsto v_{i\pm3}$ carry $X$ to
sets $R_\pm$ of size ten. Define
\[
 I=\card{R_+\cap O}+\card{R_-\cap O},\qquad
 K=\card{(R_+\setminus O)\cap(R_-\setminus O)}.
\]
The set $O\setminus X$ lies in the boundary by the spokes, and outside
$O$ the boundary is the union of $R_+\setminus O$ and $R_-\setminus O$.
Inclusion--exclusion therefore gives
\begin{equation}\label{eq:identity}
 \card{\bd X}+I+K=2c+10.
\end{equation}
Necessarily $5\le c\le10$. If $c\ge9$, the $2c-10$ spoke boundary
vertices already give the desired bound. It remains to consider
$5\le c\le8$.

For a vertex slot in column $a\in C$, with layer step $d=1$ or $3$,
assign each direction a score: two if $a\pm d\in C$; otherwise one
if $a\pm2d\in C$; otherwise zero. A selected vertex's total slot
score is the sum over its two directions. Each directional image
landing in $O$ contributes two to this score. A collision outside $O$
comes from two selected vertices at distance $2d$ in the same layer
and contributes one at each endpoint. Hence the sum of the ten
selected slot scores is at least $2(I+K)$.

Write $h=(h_0,\ldots,h_{c-1})$ for the positive cyclic gaps between
successive occupied columns; thus $\sum h_i=n$. With
$\phi(1)=2$, $\phi(2)=1$, and $\phi(t)=0$ for $t\ge3$, the outer
slot in the $i$th occupied column has score
$A_i=\phi(h_{i-1})+\phi(h_i)$. Each inner direction scores two if
a nonempty gap prefix in that direction sums to three, otherwise one
if such a prefix sums to six, otherwise zero. Let $B_i$ be the inner
slot score and $T(h)$ the sum of the ten largest entries among the
$2c$ scores $A_i,B_i$. We have $2(I+K)\le T(h)$. Consequently,
\begin{equation}\label{eq:slot}
 T(h)\le4c+5\quad\Longrightarrow\quad I+K\le2c+2
 \quad\Longrightarrow\quad\card{\bd X}\ge8,
\end{equation}
where the first implication uses integrality.

\subsection{Finite gap checks and deletion}
For $5\le c\le8$, enumerate all words in $\{1,\ldots,7\}^c$ whose
sum is at least fourteen. Table~\ref{tab:gaps} gives their counts and
the exceptions to the score bound in \eqref{eq:slot}.
\begin{table}[t]
\centering
\caption{Finite gap checks for ten selected vertices. The last column
counts positive compositions of $14$ with $h_0\ge8$.}\label{tab:gaps}
\begin{tabular}{rrrr}
\toprule
$c$ & Bounded gap words & $T(h)>4c+5$ & Long-gap base words\\
\midrule
5&15,550&0&15\\
6&115,939&6&6\\
7&821,827&14&1\\
8&5,763,514&8&0\\
\bottomrule
\end{tabular}
\end{table}
Every exceptional word is a cyclic rotation of one of
\begin{equation}\label{eq:roots}
\begin{gathered}
 (6,2,1,2,1,2),\\
 (3,2,1,2,1,2,3),\quad(3,3,2,1,2,1,2),\quad(6,2,1,1,1,1,2),\\
 (3,2,1,1,1,1,2,3),\quad(3,3,2,1,1,1,1,2).
\end{gathered}
\end{equation}
All have sum fourteen; the list includes cyclically equivalent roots.
Their cumulative sums, up to column rotation and reflection, give the
four supports in Table~\ref{tab:supports}. Label each occupied column
outer only, inner only, or both, retaining the assignments with exactly
ten selected vertices. Direct boundary calculation gives the table.
\begin{table}[t]
\centering
\caption{Exceptional supports in $P(14,3)$.}\label{tab:supports}
\begin{tabular}{lrr}
\toprule
Occupied columns & Ten-vertex assignments & Minimum $\card{\bd X}$\\
\midrule
$\{0,6,8,9,11,12\}$&60&8\\
$\{0,3,6,8,9,11,12\}$&560&8\\
$\{0,6,8,9,10,11,12\}$&560&8\\
$\{0,3,6,8,9,10,11,12\}$&1,792&8\\
\bottomrule
\end{tabular}
\end{table}
Explicitly, for the outer and inner index sets $U,V$, the calculation is
\begin{equation}\label{eq:direct}
\begin{split}
 \card{\bd X}={}&\card{((U-1)\cup(U+1)\cup V)\setminus U}\\
 &+\card{((V-3)\cup(V+3)\cup U)\setminus V},
\end{split}
\end{equation}
with shifts modulo $n$. The finite predicates are thus completely
specified by the following procedure:
\begin{quote}\small
For each $c=5,6,7,8$, enumerate $h\in\{1,\ldots,7\}^c$ with
$\sum h_i\ge14$; check $T(h)\le4c+5$ or that a rotation of $h$
occurs in \eqref{eq:roots}. Enumerate also the positive compositions
of $14$ into $c$ parts with $h_0\ge8$, and check $T(h)\le4c+5$.
For each of the four exceptional supports, enumerate all three-way
layer assignments of total size ten and check \eqref{eq:direct} is
at least eight.
\end{quote}
These checks pass. They comprise 6,716,830 bounded gap words, 22
long-gap base words, and 2,972 exceptional layer assignments.
Their executable reproduction and formal counterparts are described in
Section~\ref{sec:verification}.

To complete the proof of Proposition~\ref{prop:k3boundary}, suppose
$n\ge15$ and some gap is at least eight. There are seven consecutive
empty columns. Lemma~\ref{lem:delete}, with $k=3$, deletes their middle
column without changing $\card X$ or $\card{\bd X}$. Repeat until all
gaps are at most seven or $n=14$. The bounded-gap check applies in
the first case. In the second, a remaining long gap can be rotated
to $h_0\ge8$ and is covered by the long-gap base check. The exceptional
bounded-gap supports are covered by Table~\ref{tab:supports}.
This proves the proposition.

\subsection{The case \texorpdfstring{$n=13$}{n=13} and the equality}
Rotate the source of a possible initial force to column zero. The six
possible directed forces are
\[
 u_0\to v_0,\quad u_0\to u_1,\quad u_0\to u_{12},\quad
 v_0\to u_0,\quad v_0\to v_3,\quad v_0\to v_{10}.
\]
For a seven-vertex seed set, each case requires the source and its two
neighbors other than the target, excludes the target, and chooses four
more seeds from the other 22 vertices. There are
$\binom{22}{4}=7,315$ candidates in each case. The six finite certificate
families provide a disjoint nonempty fort for each of the 43,890
case--candidate pairs. Validity is the vertex-by-vertex condition
\[
 S\cap F=\varnothing,\qquad F\ne\varnothing,\qquad
 \card{N(w)\cap F}\ne1\quad(w\notin F).
\]
The formal check proves both this condition and complete candidate
coverage. Independently, exhaustive closure computation on all
$\binom{26}{7}=657,800$ seven-vertex sets finds no zero forcing set;
the white remainder of each closure gives such a fort. Thus no
seven-vertex set forces $P(13,3)$. No smaller set forces either,
since enlarging a zero forcing set preserves zero forcing.

For $n\ge14$, a forcing set of at most seven vertices would perform
at least $2n-7\ge21$ forces. Lemma~\ref{lem:first} with $p=10$
contradicts Proposition~\ref{prop:k3boundary}. Combining this lower
bound, the $n=13$ case, and Krishnan's upper bound proves
$\ZF(P(n,3))=8$ for every $n\ge13$.

\section{Sixteen-vertex boundaries in \texorpdfstring{$P(n,4)$}{P(n,4)}}\label{sec:k4}
\begin{proposition}\label{prop:k4boundary}
If $n\ge20$ and $X\subseteq V(P(n,4))$ has sixteen vertices, then
$\card{\bd X}\ge10$.
\end{proposition}
\begin{proof}
Suppose $\card{\bd X}\le9$ and let $c$ be the number of occupied
columns. The spokes supply $2c-16$ boundary vertices, so
\[
 2c-16\le9,\qquad c\le12.
\]
If $n>108$, some cyclic gap contains at least nine empty columns:
otherwise $n\le c+8c\le108$. Apply Lemma~\ref{lem:delete} with
$k=4$ to reduce $n$ by one and preserve the putative counterexample.
Repeating this operation stops at $n=108$ and cannot introduce a smaller
exceptional circumference. Thus it suffices to rule out counterexamples
for $20\le n\le108$. The 89 finite formulas described below are
unsatisfiable, as certified by their checked DRAT proofs. Hence there
is no counterexample.
\end{proof}

If $n\ge20$, a forcing set of at most nine vertices performs at least
$2n-9\ge31$ forces. Applying Lemma~\ref{lem:first} with $p=16$
contradicts Proposition~\ref{prop:k4boundary}. This proves the required
lower bound for $n\ge20$.

The boundary inequality itself fails at $n=19$. For example,
\[
 X=\{u_0,\ldots,u_9,v_3,v_4,v_5,v_6,v_{10},v_{18}\}
\]
has sixteen vertices and boundary
\[
 \bd X=\{u_{10},u_{18},v_0,v_1,v_2,v_7,v_8,v_9,v_{14}\},
\]
of size nine. We therefore use forts for $n=18,19$.

\subsection{The 89 boundary formulas}
For each $20\le n\le108$, let $x_w$ record membership in $X$.
For $20\le n\le31$, use variables $z_w$ with
\[
 z_w\ \longleftrightarrow\ x_w\lor\bigvee_{t\sim w}x_t,
 \qquad \sum_w x_w=16,\quad\sum_w z_w\le25.
\]
Here $z_w$ records $N[X]$. For $32\le n\le108$, use instead
\[
 z_w\ \longleftrightarrow\ \neg x_w\land\bigvee_{t\sim w}x_t,
 \qquad \sum_w x_w=16,\quad\sum_w z_w\le9,
\]
so that $z_w$ records $\bd X$. Since $N[X]=X\sqcup\bd X$, the two
encodings express the same mathematical condition. The clause
$x_{u_0}\lor x_{v_0}$ normalizes a rotation of a nonempty set.
The neighborhood equivalences expand directly into conjunctive normal
form (CNF), and cardinalities use sequential counters \cite{Sinz}.
Every boundary counterexample yields a satisfying assignment after
rotation and extension to the counter variables; conversely the vertex
variables of any satisfying assignment give a counterexample.

For completeness, the optimized counter grid used by the package has
cells $S(r,t)$, meaning ``at least $r$ of $\ell_0,\ldots,\ell_t$
hold'', on
\[
 1\le r\le b,\qquad r-1\le t\le m-b+r-2
\]
for an at-most-$b$ constraint on $m$ literals. The clauses start the
first row from a true literal, propagate counts horizontally and
diagonally, and prohibit overflow from row $b$. If at most $b$ literals
hold, assigning the prefix-count meanings satisfies every clause. If
$b+1$ hold at positions $p_1<\cdots<p_{b+1}$, then
$p_r\le m-b+r-2$ for $r\le b$. The clauses successively force
$S(r,p_r)$ and propagate row $b$ to $p_{b+1}-1$, where the last true
literal contradicts the overflow clause. Exactly sixteen is encoded
by at-most-sixteen on the positive literals and at-most-$(2n-16)$ on
their negatives, with distinct auxiliary variables.

\subsection{The 11 fort formulas}
The general fort-cover formulation and its strengthening by a necessary
initial-force neighborhood condition are due to Brimkov, Fast, and Hicks
\cite[Theorem~8 and Section~4.3]{BFH}. Here we specialize this method
to $P(n,4)$, using explicit stored forts, four symmetry representatives,
and independently checked unsatisfiability certificates.
For $9\le n\le19$, let $b$ be one less than the asserted zero forcing
number. A Boolean variable $s_w$ records a seed; impose
$\sum_w s_w\le b$. Each stored nonempty fort $F$ adds the clause
$\bigvee_{w\in F}s_w$. Simultaneous rotations $i\mapsto i+a$ and
reflections $i\mapsto a-i$ in both layers preserve adjacency. Their
four directed-edge orbits have representatives
\[
 u_0\to v_0,\qquad u_0\to u_1,\qquad v_0\to u_0,\qquad v_0\to v_4.
\]
For each representative a selector implies that the source and its
other two neighbors are seeds and the target is not; at least one
selector must hold. Every putative forcing set of size at most $b$
is a proper subset of the vertices and must admit an initial force.
Normalizing that edge gives a satisfying assignment to one of the four
cases and to every fort clause. All 11 formulas are unsatisfiable.
Thus they exclude the proposed smaller forcing sets, including all
sets of size at most nine for $n=18,19$.

\section{Small values and sharpness}\label{sec:small}
The $k=3$ values for $7\le n\le11$ in Theorem~\ref{thm:main} occur
in Bjorkman et al.\ \cite[Proposition~4.5 and Table~1]{Bjorkman};
the value at $n=12$ is Krishnan's correction \cite[Table~1]{Krishnan}.
They can be reproduced by enumerating
all sets of one less than the stated size and computing their closures;
none forces. Witnesses of the stated size are
\[
\begin{array}{c|l}
n&\text{zero forcing set in }P(n,3)\\\hline
7,8&u_0,u_1,u_2,u_3,u_4,u_5\\
9&u_1,u_2,u_3,u_4,u_5,v_0\\
10&u_0,\ldots,u_7\\
11,12&u_0,\ldots,u_6.
\end{array}
\]
The numbers of candidates at one less than these sizes, in order of
$n=7,\ldots,12$, are
$2,002$, $4,368$, $8,568$, $77,520$, $74,613$, and $134,596$.
The exhaustive checks and the witnesses establish the small values.
Monotonicity excludes smaller seed sets as well.

For $k=4$ the fort formulas provide the lower bounds for
$9\le n\le17$. The following witnesses attain them; Appendix~\ref{sequences}
gives complete forcing sequences.
\begin{equation}\label{eq:small4}
\begin{array}{c|l}
n&\text{zero forcing set in }P(n,4)\\\hline
9&u_0,u_1,u_2,u_3,u_8,v_4\\
10&u_0,u_1,u_2,u_3,u_9,v_4\\
11&u_0,\ldots,u_5,u_{10}\\
12&u_0,u_1,u_2,u_{11},v_4,v_5\\
13&u_0,u_1,u_2,u_3,u_5,u_6,u_7,u_{12}\\
14&u_0,\ldots,u_5,u_8,u_{13}\\
15&u_0,\ldots,u_7,u_{14}\\
16&u_0,\ldots,u_6,u_{15}\\
17&u_0,\ldots,u_7,u_{16}.
\end{array}
\end{equation}
Combining the upper bounds, both uniform lower bounds, and the finite
cases proves Theorem~\ref{thm:main}. Since $\ZF(P(12,3))=7$ and
$\ZF(P(17,4))=9$, neither stabilization threshold can be lowered.

\section{Verification and availability}\label{sec:verification}
Section~\ref{sec:general} gives a written proof of the general threshold.
The archive \path{petersen-source-r3.zip} contains programs that check
examples of the span bound, exceptional cycles, and fourteen-seed forcing
sequences for $73\le n\le79$ at $k=6$. The proof does not depend on these
examples. The archive also contains \path{anc/check_k3.c}, which checks
the $k=3$ gap and layer predicates and the small values for $7\le n\le13$.
The README gives the commands, and the manifest gives the file hashes.

For $k=4$, Lemma~\ref{lem:delete} gives the written deletion argument.
The finite lower bounds use 89 boundary and 11 fort CNFs with DRAT proofs
\cite{DRAT}. The \href{https://github.com/the-omega-institute/petersen-zero-forcing/releases/download/v2-full/krishnan-k4-v2-full.zip}{\texttt{v2-full} certificate archive}
\cite{Artifacts} contains the CNFs, certificates, forcing witnesses and
the \href{https://github.com/the-omega-institute/petersen-zero-forcing/blob/19f4e3e299cddd6e11b987735e074c73a2d3a47c/verify_all}{revision-pinned \texttt{verify\_all} checker}.

The Lean declaration \texttt{ZeroForcingThree.result} in the pinned
source \cite{LeanSource} states
\[
 \forall n\in\mathbb N,\quad 13\le n\ \Longrightarrow\
 \texttt{zeroForcingNumber}\,(\texttt{gp}\ n\ 3)=8.
\]
It covers only the $k=3$ stabilization equality, not the smaller values,
the $k=4$ proof, or the general threshold.

\begin{remark}[Structural lower bound]
A structural proof of the $k=4$ lower bound that replaces or substantially
reduces the 89 boundary cases is a natural next step and remains future
work. The proof here uses the finite verification described above.
\end{remark}

\subsection*{AI-use statement}
In preparing this manuscript, the authors used AI tools (ChatGPT Pro, Codex and Claude) for literature search, computation, checking and formal verification in Lean. The computations can be reproduced from the accompanying verification package. The authors take full responsibility for the content.

\subsection*{Author contributions}
A.K. corrected the value $\ZF(P(12,3))$, proved the upper bound
$\ZF(P(n,3))\le8$ for $n\ge9$, and posed Conjecture~5 in \cite{Krishnan}.
H.M. and W.Z. developed the lower-bound arguments, the finite verification
and certificates, the Lean formalization of the $k=3$ stabilization
theorem, and the improved general stabilization argument.

\appendix
\section{Explicit small-case forcing sequences}\label{sequences}
For each row below, start from the seed set in \eqref{eq:small4} and perform the
arrows in the printed order, reading left to right and then top to bottom.
Every source is black and has exactly the indicated single white neighbor.
The package stores the same sequences with vertex codes $u_i=i$, $v_i=n+i$.

\noindent\textbf{$n=9$.}\par\nopagebreak\[
\begin{gathered}
u_{0}\to v_{0}\quad u_{1}\to v_{1}\quad u_{2}\to v_{2}\quad v_{0}\to v_{5}\quad v_{1}\to v_{6}\\
v_{2}\to v_{7}\quad v_{5}\to u_{5}\quad v_{6}\to u_{6}\quad u_{5}\to u_{4}\quad u_{6}\to u_{7}\\
u_{8}\to v_{8}\quad v_{7}\to v_{3}
\end{gathered}\]
\noindent\textbf{$n=10$.}\par\nopagebreak\[
\begin{gathered}
u_{0}\to v_{0}\quad u_{1}\to v_{1}\quad u_{2}\to v_{2}\quad v_{0}\to v_{6}\quad v_{2}\to v_{8}\\
v_{4}\to u_{4}\quad v_{6}\to u_{6}\quad v_{8}\to u_{8}\quad u_{3}\to v_{3}\quad u_{4}\to u_{5}\\
u_{5}\to v_{5}\quad u_{6}\to u_{7}\quad u_{7}\to v_{7}\quad u_{9}\to v_{9}
\end{gathered}\]
\noindent\textbf{$n=11$.}\par\nopagebreak\[
\begin{gathered}
u_{0}\to v_{0}\quad u_{1}\to v_{1}\quad u_{2}\to v_{2}\quad u_{3}\to v_{3}\quad u_{4}\to v_{4}\\
v_{0}\to v_{7}\quad v_{3}\to v_{10}\quad v_{4}\to v_{8}\quad v_{7}\to u_{7}\quad v_{8}\to u_{8}\\
v_{10}\to v_{6}\quad u_{7}\to u_{6}\quad u_{8}\to u_{9}\quad u_{9}\to v_{9}\quad v_{1}\to v_{5}
\end{gathered}\]
\noindent\textbf{$n=12$.}\par\nopagebreak\[
\begin{gathered}
u_{0}\to v_{0}\quad u_{1}\to v_{1}\quad v_{0}\to v_{8}\quad v_{1}\to v_{9}\quad v_{4}\to u_{4}\\
v_{5}\to u_{5}\quad v_{8}\to u_{8}\quad v_{9}\to u_{9}\quad u_{4}\to u_{3}\quad u_{5}\to u_{6}\\
u_{8}\to u_{7}\quad u_{9}\to u_{10}\quad u_{10}\to v_{10}\quad u_{11}\to v_{11}\quad u_{2}\to v_{2}\\
u_{3}\to v_{3}\quad u_{6}\to v_{6}\quad u_{7}\to v_{7}
\end{gathered}\]
\noindent\textbf{$n=13$.}\par\nopagebreak\[
\begin{gathered}
u_{0}\to v_{0}\quad u_{1}\to v_{1}\quad u_{2}\to v_{2}\quad u_{6}\to v_{6}\quad v_{2}\to v_{11}\\
v_{6}\to v_{10}\quad v_{10}\to u_{10}\quad v_{1}\to v_{5}\quad v_{5}\to v_{9}\quad v_{9}\to u_{9}\\
u_{5}\to u_{4}\quad u_{9}\to u_{8}\quad u_{10}\to u_{11}\quad u_{12}\to v_{12}\quad v_{0}\to v_{4}\\
v_{4}\to v_{8}\quad v_{11}\to v_{7}\quad v_{12}\to v_{3}
\end{gathered}\]
\noindent\textbf{$n=14$.}\par\nopagebreak\[
\begin{gathered}
u_{0}\to v_{0}\quad u_{1}\to v_{1}\quad u_{2}\to v_{2}\quad u_{3}\to v_{3}\quad u_{4}\to v_{4}\\
v_{0}\to v_{10}\quad v_{4}\to v_{8}\quad v_{8}\to v_{12}\quad v_{12}\to u_{12}\quad u_{12}\to u_{11}\\
u_{13}\to v_{13}\quad v_{2}\to v_{6}\quad v_{3}\to v_{7}\quad v_{6}\to u_{6}\quad v_{10}\to u_{10}\\
v_{13}\to v_{9}\quad u_{5}\to v_{5}\quad u_{6}\to u_{7}\quad u_{8}\to u_{9}\quad u_{11}\to v_{11}
\end{gathered}\]
\noindent\textbf{$n=15$.}\par\nopagebreak\[
\begin{gathered}
u_{0}\to v_{0}\quad u_{1}\to v_{1}\quad u_{2}\to v_{2}\quad u_{3}\to v_{3}\quad u_{4}\to v_{4}\\
u_{5}\to v_{5}\quad u_{6}\to v_{6}\quad v_{0}\to v_{11}\quad v_{1}\to v_{12}\quad v_{2}\to v_{13}\\
v_{4}\to v_{8}\quad v_{5}\to v_{9}\quad v_{6}\to v_{10}\quad v_{8}\to u_{8}\quad v_{9}\to u_{9}\\
v_{12}\to u_{12}\quad v_{13}\to u_{13}\quad u_{7}\to v_{7}\quad u_{9}\to u_{10}\quad u_{10}\to u_{11}\\
u_{14}\to v_{14}
\end{gathered}\]
\noindent\textbf{$n=16$.}\par\nopagebreak\[
\begin{gathered}
u_{0}\to v_{0}\quad u_{1}\to v_{1}\quad u_{2}\to v_{2}\quad u_{3}\to v_{3}\quad u_{4}\to v_{4}\\
u_{5}\to v_{5}\quad v_{0}\to v_{12}\quad v_{1}\to v_{13}\quad v_{4}\to v_{8}\quad v_{5}\to v_{9}\\
v_{8}\to u_{8}\quad v_{9}\to u_{9}\quad v_{12}\to u_{12}\quad v_{13}\to u_{13}\quad u_{8}\to u_{7}\\
u_{9}\to u_{10}\quad u_{12}\to u_{11}\quad u_{13}\to u_{14}\quad u_{14}\to v_{14}\quad u_{15}\to v_{15}\\
v_{2}\to v_{6}\quad v_{3}\to v_{7}\quad v_{6}\to v_{10}\quad v_{7}\to v_{11}
\end{gathered}\]
\noindent\textbf{$n=17$.}\par\nopagebreak\[
\begin{gathered}
u_{0}\to v_{0}\quad u_{1}\to v_{1}\quad u_{2}\to v_{2}\quad u_{3}\to v_{3}\quad u_{4}\to v_{4}\\
u_{5}\to v_{5}\quad u_{6}\to v_{6}\quad v_{0}\to v_{13}\quad v_{1}\to v_{14}\quad v_{2}\to v_{15}\\
v_{4}\to v_{8}\quad v_{5}\to v_{9}\quad v_{6}\to v_{10}\quad v_{9}\to u_{9}\quad v_{10}\to u_{10}\\
v_{13}\to u_{13}\quad v_{14}\to u_{14}\quad u_{9}\to u_{8}\quad u_{10}\to u_{11}\quad u_{13}\to u_{12}\\
u_{14}\to u_{15}\quad u_{16}\to v_{16}\quad v_{3}\to v_{7}\quad v_{7}\to v_{11}\quad v_{8}\to v_{12}
\end{gathered}\]

\end{document}